\documentclass{article}
\usepackage{amssymb,amsmath}
\usepackage[mathscr]{euscript}
\usepackage[left=3cm,right=3cm]{geometry}
\title{Fixed Points for \'Ciri\'c-$G$-Contractions in Uniform Spaces Endowed with a Graph\\[0.3cm]}
\author{{Aris Aghanians$^1$,\,\,\,Kamal Fallahi$^1$,\,\,\,Kourosh Nourouzi$^1$\thanks{Corresponding
author} \thanks{e-mail: nourouzi@kntu.ac.ir; fax: +98 21
22853650},\,\,\,Ram U Verma$^2$}\\[0.4cm]
{\em $^1$Department of Mathematics, K. N. Toosi University of Technology,}\\
{\em P.O. Box 16315-1618, Tehran, Iran.}\\[0.4cm]
\and
{\em $^{2}$ International Publications USA,}\\
{\em 3400 S Brahma Blvd Suite 31B, Kingsville, TX 78363 - 7342, USA}\\}
\newenvironment{proof}{\noindent{\em{Proof.}}}{$\hfill\square$
\medskip}
\newtheorem{defn}{Definition}
\newtheorem{cor}{Corollary}
\newtheorem{exm}{Example}
\newtheorem{thm}{Theorem}
\newtheorem{prop}{Proposition}
\newtheorem{rem}{Remark}
\newtheorem{lem}{Lemma}

\newcommand{\fix}{{\rm Fix}}
\begin{document}
\maketitle \begin{abstract} In this paper, we generalize the
notion  of $\lambda$-generalized contractions introduced by \'Ciri\'c from  metric to  uniform
spaces endowed with a graph and discuss on the existence and
uniqueness of fixed points for this type of contractions using the basic entourages.
\end{abstract}
\def\thefootnote{ \ }
\footnotetext{{\em}$2010$ Mathematics Subject Classification.
47H10, 05C40.\par {\bf Keywords:} Separated uniform space;
\'Ciri\'c-$G$-contraction; Fixed point.}

\section{Introduction and Preliminaries}
In \cite{cir}, \'Ciri\'c  introduced the notion of a $\lambda$-generalized contraction on a metric
space $X$ as follows:
\begin{eqnarray*}
d(Tx,Ty)&\leq& q(x,y)d(x,y)+r(x,y)d(x,Tx)+s(x,y)d(y,Ty)\cr\\[-.3cm] &\
&+\,t(x,y)\big(d(x,Ty)+d(y,Tx)\big)\qquad\big((x,y)\in X),
\end{eqnarray*}
where $q$, $r$, $s$ and $t$ are nonnegative functions on $X\times
X$ such that
$$\sup\big\{q(x,y)+r(x,y)+s(x,y)+2t(x,y):x,y\in
X\big\}=\lambda<1.$$\par
Acharya \cite{ach} investigated some well-known types of
contractions in uniform spaces and Rhoades \cite{rho}
discussed on $\lambda$-generalized type contractions in uniform
spaces.\par Recently, Jachymski \cite{jac} entered graphs
in metric fixed point theory and generalized the Banach contraction principle in
both metric and partially ordered metric spaces.
For further works and results in metric and uniform spaces endowed with a
graph, see, e.g., \cite{agh, boj10,boj12,pet}.\par In this paper,  we generalize the
notion  of $\lambda$-generalized contractions from  metric to  uniform
spaces endowed with a graph and establish some results  on the existence and
uniqueness of fixed points via an entourage approach for this type of contractions. Despite the method  given in \cite{boj12} that the results therein may not be applied for (the partially ordered contractions induced by graph) their partially ordered counterparts, we will see that
%without putting the assumption of  $T$-connectedness for the graph of uniform space,
 our contractions  both are extensions
of \'Ciri\'c-Reich-Rus operators  and  may also be converted to the language of partially ordered metric or uniform spaces.
  %$\clubsuit$ Recently in \cite{boj12}, the author proved some fixed point results in metric spaces endowed with a graph using the notion of $T$-connectedness. But this work cannot generalize the same results in partially ordered metric spaces because $T$-connectedness implies that the order is total, that is, each two points are comparable and hence the contraction is not an order contraction. But we prove the existence and uniqueness of fixed points without using the concept.$\clubsuit$
  \par We start by reviewing a few basic notions in uniform
spaces. For a widespread discussion on the uniform spaces, the reader can see, e.g.,
\cite[pp.238-277]{wil}.\par Suppose that $X$ is a nonempty set and
$U$ and $V$ are nonempty subsets of $X\times X$. We let
\begin{itemize}
\item $\Delta(X)=\{(x,x):x\in X\}$ be the diagonal of $X$; \item
$U^{-1}=\{(x,y):(y,x)\in U\}$ be the inverse of $U$; and \item
$U\circ V=\{(x,y):\exists\,z\in X\ {\rm s.t.}\ (x,z)\in V\
,(z,y)\in U\}$.
\end{itemize}
Now assume that $\mathscr U$ is a nonempty family of subsets of
$X\times X$ satisfying the following properties:
\begin{description}
\item[\rm U1)] Each member of $\mathscr U$ contains
$\Delta(X)$; \item[\rm U2)] The intersection of each two members of
$\mathscr U$ lies in $\mathscr U$; \item[\rm U3)] $\mathscr U$
contains the inverses of its members; \item[\rm U4)] For each
$U\in\mathscr U$, there exists a $V\in\mathscr U$ such that
$V\circ V\subseteq U$; \item[\rm U5)] If $U\in\mathscr U$ and
$U\subseteq V$, then $V\in\mathscr U$.
\end{description}
Then $\mathscr U$ is called a uniformity on $X$ and the pair
$(X,\mathscr U)$ (shortly denoted by $X$) is called a uniform
space.\par It is well-known that a uniformity $\mathscr U$ on a
set $X$ is separating if the intersection of all members of
$\mathscr U$ is exactly the diagonal $\Delta(X)$. If this
is satisfied, then $X$ is called a separated uniform
space.\par To remind the convergence and Cauchyness notions in uniform spaces, let $\{x_n\}$ be a sequence in a uniform space $X$. Then $\{x_n\}$ is said to
be convergent to a point $x\in X$, denoted by $x_n\rightarrow x$,
if for each $U\in\mathscr U$, there exists an $N>0$ such that
$(x_n,x)\in U$ for all $n\geq N$, and it is said to be Cauchy
in $X$ if for each $U\in\mathscr U$, there exists an
$N>0$ such that $(x_m,x_n)\in U$ for all $m,n\geq N$. The uniform
space $X$ is called sequentially complete if each Cauchy sequence
in $X$ is convergent to some point of $X$. It can be easily
verified that if $x_n\rightarrow x$, then each subsequence of $\{x_n\}$ converges to $x$, and further in a separated uniform space, each sequence may
converge to at most one point, i.e., the limits of convergent
sequences is unique in separated uniform spaces.\par  Let $\mathscr F$ be a nonempty collection
of (uniformly continuous) pseudometrics on $X$ that generates the
uniformity $\mathscr U$ (see, \cite[Theorem 2.1]{ach}), and denote
by $\mathscr V$, the family of all sets of the form
$$\bigcap_{i=1}^m\Big\{(x,y)\in X\times X:\rho_i(x,y)<r_i\Big\},$$
where $m$ is a positive integer, $\rho_i\in\mathscr F$ and $r_i>0$
for $i=1,\ldots,m$. Then it has been shown that $\mathscr V$ is a
base for the uniformity $\mathscr U$, i.e., $\mathscr V$ satisfies (U1)-(U4) and each member of
$\mathscr U$ contains at least one member of $\mathscr V$. Finally, if
$$V=\bigcap_{i=1}^m\Big\{(x,y)\in X\times X:\rho_i(x,y)<r_i\Big\}\in\mathscr V$$ and $\alpha>0$, then
the set
$$\alpha V =\bigcap_{i=1}^m\Big\{(x,y)\in X\times X:\rho_i(x,y)<\alpha r_i\Big\}$$
is still a member of $\mathscr V$.\par The next lemma embodies
some important properties about the above-mentioned sets. For
other properties, the reader is referred to \cite[Lemmas
2.1-2.6]{ach}.

\begin{lem}{\rm\cite{ach}}\label{lem1} Let $X$ be a uniform space
and $\mathscr V$ be as above. Then the following assertions hold.
\begin{description}
\item[\rm i)] If $0<\alpha\leq\beta$, then $\alpha V\subseteq\beta
V$ for all $V\in\mathscr V$. \item[\rm ii)] If $\alpha,\beta>0$, then $\alpha V\circ\beta
V\subseteq(\alpha+\beta)V$ for all $V\in\mathscr V$. \item[\rm iii)] For each $x,y\in X$ and each $V\in\mathscr V$,
there exists a positive number $\lambda$ such that
$(x,y)\in\lambda V$. \item[\rm iv)] For each $V\in\mathscr V$,
there exists a pesudometric $\rho$ on $X$ such that $(x,y)\in V$ if and only if
$\rho(x,y)<1$.
\end{description}
\end{lem}

\begin{rem}\rm\label{remark1}
The pseudometric $\rho$ in Lemma \ref{lem1} (iv) is called the Minkowski's psudometric of $V$. Moreover, for any $\alpha>0$, we have $(x,y)\in\alpha V$ if and only if $\rho(x,y)<\alpha$. In other words, $\frac1\alpha\rho$ is the Minkowski's pseudometric of $\alpha V$.
\end{rem}

\section{Main Results}
Throughout this section, the letter $X$ is used to denote a
nonempty set equipped with a uniformity $\mathscr U$ and
$\mathscr{F}$ is a nonempty collection of (uniformly continuous)
pseudometrics on $X$ generating the uniformity $\mathscr{U}$.
Furthermore, $\mathscr{V}$ is the collection of all sets of the
form
$$\bigcap_{i=1}^m\Big\{(x,y)\in X\times X:\rho_i(x,y)<r_i\Big\},$$
where $m$ is a positive integer, $\rho_i\in\mathscr{F}$ and
$r_i>0$ for $i=1,\ldots,m$. The uniform space $X$ is also endowed with
a directed graph $G$ without any parallel edges such that $V(G)=X$
and $E(G)\supseteq\Delta(X)$, i.e., $E(G)$ contains all loops, and
by $\widetilde G$, it is meant the undirected graph obtained from
$G$ by ignoring the directions of the edges of $G$. The set of all
fixed points of a mapping $T:X\rightarrow X$ is denoted by
$\fix(T)$ and we set  $X_T=\{x\in X:(x,Tx)\in
E(G)\}$.\par The idea of following  definition  is taken from \cite[2.1. Definition]{cir} and
\cite[Definition 2.1]{jac} and it generalizes both of them.

\begin{defn}\label{ciric}\rm
Let $T$ be a mapping of $X$ into itself. Then we call $T$ a
\'Ciri\'c-$G$-contraction if
\begin{description}
\item[\rm C1)] $(x,y)\in E(G)$ implies $(Tx,Ty)\in E(G)$ for all
$x,y\in X$, that is, $T$ is edge preserving; \item[\rm C2)] for all
$x,y\in X$ and all $V_1,V_2,V_3,V_4,V_5\in\mathscr V$,
$$(x,y)\in E(G)\cap V_1,\ (x,Tx)\in V_2,\ (y,Ty)\in V_3,\ (x,Ty)\in V_4,\ \mbox{and}\ (y,Tx)\in V_5$$ imply
$$(Tx,Ty)\in a_1(x,y)V_1\circ a_2(x,y)V_2\circ a_3(x,y)V_3\circ a_4(x,y)V_4\circ a_4(x,y)V_5,$$
where $a_1$, $a_2$, $a_3$ and $a_4$ are positive-valued functions on
$X\times X$ satisfying
\begin{equation}\label{supremum}
\sup\big\{a_1(x,y)+a_2(x,y)+a_3(x,y)+2a_4(x,y):x,y\in
X\big\}=\alpha<1.
\end{equation}
\end{description}
\end{defn}\par
Note that if (\ref{supremum}) holds, then
$$a_1(x,y)+a_2(x,y)+a_4(x,y)+\alpha\big(a_3(x,y)+a_4(x,y)\big)<a_1(x,y)+a_2(x,y)+a_3(x,y)+2a_4(x,y)\leq\alpha,$$
for all $x,y\in X$. So
$$a_1(x,y)+a_2(x,y)+a_4(x,y)<\alpha\big(1-a_3(x,y)-a_4(x,y)\big)\qquad(x,y\in X),$$
which yields
\begin{equation}\label{abc}
\frac{a_1(x,y)+a_2(x,y)+a_4(x,y)}{1-a_3(x,y)-a_4(x,y)}<\alpha\qquad(x,y\in
X).
\end{equation}
%We now give a few basic examples of \'Ciri\'c-$G$-contractions on
%uniform spaces endowed with a graph.

\begin{exm}\rm
\begin{enumerate}
\item Since $E(G)$ and each member $V$ of $\mathscr V$ contain
$\Delta(X)$, it follows that each constant mapping $T:X\rightarrow
X$ is a \'Ciri\'c-$G$-contraction with any positive-valued functions
$a_1$, $a_2$, $a_3$ and $a_4$ satisfying \eqref{supremum}. \item
Let $G_0$ be the complete graph with $V(G_0)=X$, i.e.,
$E(G_0)=X\times X$. Then \'Ciri\'c-$G_0$-contractions
(simply \'Ciri\'c-contractions) are precisely the contractions that
generalize $\lambda$-generalized contractive mappings
introduced by \'Ciri\'c in \cite[2.1. Definition]{cir} (the existence and uniqueness of fixed
points for this type of contractions on sequentially complete
separated uniform spaces were investigated by Rhoades
\cite[Theorem 1]{rho}). \item Let $\preceq$ be a partial order on
$X$, and consider a graph $G_1$ by $V(G_1)=X$ and
$$E(G_1)=\big\{(x,y)\in X\times X:x\preceq y\big\}.$$
Then $E(G_1)$ contains all loops and \'Ciri\'c-$G_1$-contractions
are precisely the nondecreasing order \'Ciri\'c contractions.
\end{enumerate}
\end{exm}

Since each member of $\mathscr V$ is symmetric, the next
proposition follows immediately from the definition of  \'Ciri\'c-$G$-contraction.

\begin{prop}\label{tilde} If a mapping from $X$ into itself satisfies {\rm(C1)}
(respectively, {\rm(C2)}) for a graph $G$, then it satisfies
{\rm(C1)} (respectively, {\rm(C2)}) for the graph $\widetilde G$.
In particular, a \'Ciri\'c-$G$-contraction is also a
\'Ciri\'c-$\widetilde G$-contraction.
\end{prop}

To investigate the existence and uniqueness of fixed points for \'Ciri\'c-$G$-contractions, we need following lemmas.

\begin{lem}\label{lem2} Let $T:X\rightarrow X$ be a \'Ciri\'c-$G$-contraction and $V\in\mathscr V$. If $x\in X_T$ is such that $(x,Tx)\in V$, then
$$(T^nx,T^{n+1}x)\in\alpha^nV\qquad n=0,1,\ldots\,,$$
where $\alpha$ is as in \eqref{supremum}.
\end{lem}

\begin{proof} If $n=0$, then there is nothing to prove.  Let $n\geq1$ and denote by $\rho$, the Minkowski's pseudometric of $V$. Write
$$\rho(T^{n-1}x,T^nx)=r_1,\quad\rho(T^nx,T^{n+1}x)=r_2,\quad\mbox{and}\quad\rho(T^{n-1}x,T^{n+1}x)=r_3$$
and let $\varepsilon>0$. Then it is clear that
$$(T^{n-1}x,T^nx)\in(r_1+\varepsilon)V,\quad(T^nx,T^{n+1}x)\in(r_2+\varepsilon)V,$$
$$(T^{n-1}x,T^{n+1}x)\in(r_3+\varepsilon)V,\quad\mbox{and}\quad(T^nx,T^nx)\in\varepsilon V.$$
Note that by (C1), we have $(T^{n-1}x,T^nx)\in E(G)$. Hence it follows by (C2) and Lemma \ref{lem1} that
\begin{eqnarray*}
(T^nx,T^{n+1}x)&\in&a_1(T^{n-1}x,T^nx)(r_1+\varepsilon)V\circ
a_2(T^{n-1}x,T^nx)(r_1+\varepsilon)V\cr\\[-.3cm]
&\ &\circ\,a_3(T^{n-1}x,T^nx)(r_2+\varepsilon)V\circ
a_4(T^{n-1}x,T^nx)(r_3+\varepsilon)V\circ a_4(T^{n-1}x,T^nx)\varepsilon V\cr\\[-.3cm]
&\subseteq&\Big(\big(a_1(T^{n-1}x,T^nx)+a_2(T^{n-1}x,T^nx)\big)r_1
+a_3(T^{n-1}x,T^nx)r_2+a_4(T^{n-1}x,T^nx)r_3\cr\\[-.3cm]
&\ &+\,\big(a_1(T^{n-1}x,T^nx)+a_2(T^{n-1}x,T^nx)+a_3(T^{n-1}x,T^nx)
+2a_4(T^{n-1}x,T^nx)\big)\varepsilon\Big)V\cr\\[-.3cm]
&\subseteq&\Big(\big(a_1(T^{n-1}x,T^nx)+a_2(T^{n-1}x,T^nx)\big)r_1\cr\\[-.3cm]
&\
&+\,a_3(T^{n-1}x,T^nx)r_2+a_4(T^{n-1}x,T^nx)r_3+\alpha\varepsilon\Big)V,
\end{eqnarray*}
where $\alpha$ is as in \eqref{supremum}. Because $\rho$ is the
Minkowski's pseudometric of $V$, it follows by Remark \ref{remark1} that
\begin{eqnarray*}
\rho(T^nx,T^{n+1}x)&<&\big(a_1(T^{n-1}x,T^nx)+a_2(T^{n-1}x,T^nx)\big)r_1\cr\\[-.3cm]
&\ &+\,a_3(T^{n-1}x,T^nx)r_2
+a_4(T^{n-1}x,T^nx)r_3+\alpha\varepsilon\cr\\[-.3cm]
&=&\big(a_1(T^{n-1}x,T^nx)+a_2(T^{n-1}x,T^nx)\big)\rho(T^{n-1}x,T^nx)\cr\\[-.3cm]
&\
&+\,a_3(T^{n-1}x,T^nx)\rho(T^nx,T^{n+1}x)+a_4(T^{n-1}x,T^nx)\rho(T^{n-1}x,T^{n+1}x)+\alpha\varepsilon.
\end{eqnarray*}
Since $\varepsilon>0$ was arbitrary, we obtain
\begin{eqnarray*}
\rho(T^nx,T^{n+1}x)&\leq&\big(a_1(T^{n-1}x,T^nx)+a_2(T^{n-1}x,T^nx)\big)\rho(T^{n-1}x,T^nx)\cr\\[-.3cm]
&\
&+\,a_3(T^{n-1}x,T^nx)\rho(T^nx,T^{n+1}x)+a_4(T^{n-1}x,T^nx)\rho(T^{n-1}x,T^{n+1}x)\cr\\[-.3cm]
&\leq&\big(a_1(T^{n-1}x,T^nx)+a_2(T^{n-1}x,T^nx)\big)\rho(T^{n-1}x,T^nx)\cr\\[-.3cm]
&\
&+\,a_3(T^{n-1}x,T^nx)\rho(T^nx,T^{n+1}x)\cr\\[-.3cm]
&\ &+\,a_4(T^{n-1}x,T^nx)\big(\rho(T^{n-1}x,T^nx)+\rho(T^nx,T^{n+1}x)\big).
\end{eqnarray*}
Therefore, by (\ref{abc}),
\begin{eqnarray}\label{inductiononn}
\rho(T^nx,T^{n+1}x)&\leq&\frac{a_1(T^{n-1}x,T^nx)+a_2(T^{n-1}x,T^nx)+a_4(T^{n-1}x,T^nx)}
{1-a_3(T^{n-1}x,T^nx)-a_4(T^{n-1}x,T^nx)}
\cdot\rho(T^{n-1}x,T^nx)\nonumber\cr\\
&<&\alpha\rho(T^{n-1}x,T^nx)<\cdots<\alpha^n\rho(x,Tx).
\end{eqnarray}
Because $(x,Tx)\in V$, it follows that $\rho(x,Tx)<1$, and hence
using (\ref{inductiononn}), one has
$\rho(T^nx,T^{n+1}x)<\alpha^n$, that is,
$(T^nx,T^{n+1}x)\in\alpha^nV$.
\end{proof}

\begin{lem}\label{lem3} Let $T:X\rightarrow X$ be a \'Ciri\'c-$G$-contraction.
Then the sequence $\{T^nx\}$ is Cauchy in $X$ for all $x\in X_T$.
\end{lem}

\begin{proof} Let $x\in X_T$ and $V\in\mathscr V$ be given. Then Lemma \ref{lem1} ensures the existence of a positive number $\lambda$
such that $(x,Tx)\in\lambda V$, and so, by Lemma \ref{lem2} we have
$$(T^nx,T^{n+1}x)\in(\alpha^n\lambda)V\qquad n=0,1,\ldots,$$
where $\alpha$ is as in \eqref{supremum}. Now, if $\rho$ is the Minkowski's pseudometric of $V$, then by Remark \ref{remark1},
$\rho(T^nx,T^{n+1}x)<\alpha^n\lambda$ for all $n\geq0$, and since $\alpha<1$, it follows that
$$\sum_{n=0}^\infty\rho(T^nx,T^{n+1}x)\leq\sum_{n=0}^\infty\alpha^n\lambda=\frac\lambda{1-\alpha}<\infty.$$
An easy argument shows that $\rho(T^mx,T^nx)\rightarrow0$ as
$m,n\rightarrow\infty$. Hence there exists an $N>0$ such that
$\rho(T^mx,T^nx)<1$ for all $m,n\geq N$. Therefore, $(T^mx,T^nx)\in V$ for
all $m,n\geq N$, and because $V\in\mathscr V$ was arbitrary, it is
concluded that the sequence $\{T^nx\}$ is Cauchy in $X$.
\end{proof}

We are now ready to prove our main theorem.

\begin{thm}\label{main}
Suppose that the uniform space $X$ is sequentially complete and
separated, and has the following property:
\begin{description}
\item[$(\ast)$] If a sequence $\{x_n\}$ converges to some point
$x\in X$ and it satisfies $(x_n,x_{n+1})\in E(G)$ for all
$n\geq1$, then there exists a subsequence $\{x_{n_k}\}$ of
$\{x_n\}$ such that $(x_{n_k},x)\in E(G)$ for all $k\geq1$.
\end{description}
Then a \'Ciri\'c-$G$-contraction $T:X\rightarrow X$ has a fixed
point if and only if $X_T\neq\emptyset$. Furthermore, this fixed
point is unique if
\begin{description}
\item[\rm i)] the functions $a_2$ and $a_3$ in {\rm(C2)} coincide
on $X\times X$; and \item[\rm ii)] for all $x,y\in X$, there
exists a $z\in X$ such that $(x,z),(y,z)\in E(\widetilde G)$.
\end{description}
\end{thm}

\begin{proof} It is clear that each fixed point of $T$ is an element of $X_T$. For the converse, let $x\in X_T$. Then by Lemma \ref{lem3}, the
sequence $\{T^nx\}$ is Cauchy in $X$. By sequential completeness of $X$, there exists an $x^*\in X$ such that $T^nx\rightarrow
x^*$. On the other hand, since $x\in X_T$ and $T$ is edge preserving, it follows that
$(T^nx,T^{n+1}x)\in E(G)$ for all $n\geq0$. Therefore, by Property
$(\ast)$, there exists a strictly increasing sequence $\{n_k\}$ of
positive integers such that $(T^{n_k}x,x^*)\in E(G)$ for all $k\geq1$. We shall
show that $T^{n_k+1}x\rightarrow Tx^*$. To this end, let $V$ be an
arbitrary member of $\mathscr V$ and denote by $\rho$, the Minkowski's pseudometric of
$V$. Let $k\geq1$, write
$$\rho(T^{n_k}x,x^*)=r_1,\quad\rho(T^{n_k}x,T^{n_k+1}x)=r_2,\quad\rho(x^*,Tx^*)=r_3,$$
$$\rho(T^{n_k}x,Tx^*)=r_4,\quad\mbox{and}\quad\rho(x^*,T^{n_k+1}x)=r_5,$$ and take $\varepsilon>0$. Then it is clear that
$$(T^{n_k}x,x^*)\in(r_1+\varepsilon)V,\quad(T^{n_k}x,T^{n_k+1}x)\in(r_2+\varepsilon)V,
\quad(x^*,Tx^*)\in(r_3+\varepsilon)V,$$
$$(T^{n_k}x,Tx^*)\in(r_4+\varepsilon)V,\quad\mbox{and}\quad(x^*,T^{n_k+1}x)\in(r_5+\varepsilon)V.$$
Therefore, by (C2) and Lemma \ref{lem1}, we have
\begin{eqnarray*}
(T^{n_k+1}x,Tx^*)&\in&a_1(T^{n_k}x,x^*)(r_1+\varepsilon)V\circ
a_2(T^{n_k}x,x^*)(r_2+\varepsilon)V\circ
a_3(T^{n_k}x,x^*)(r_3+\varepsilon)V\cr\\[-.3cm] &\ &\circ\,
a_4(T^{n_k}x,x^*)(r_4+\varepsilon)V\circ
a_4(T^{n_k}x,x^*)(r_5+\varepsilon)V\cr\\[-.3cm] &\subseteq&
\Big(a_1(T^{n_k}x,x^*)r_1+a_2(T^{n_k}x,x^*)r_2+a_3(T^{n_k}x,x^*)r_3+a_4(T^{n_k}x,x^*)\big(r_4+r_5\big)\cr\\[-.3cm]
&\
&+\,\big(a_1(T^{n_k}x,x^*)+a_2(T^{n_k}x,x^*)+a_3(T^{n_k}x,x^*)+2a_4(T^{n_k}x,x^*)\big)\varepsilon\Big)V\cr\\[-.3cm]
&\subseteq&\Big(a_1(T^{n_k}x,x^*)r_1+a_2(T^{n_k}x,x^*)r_2+a_3(T^{n_k}x,x^*)r_3+a_4(T^{n_k}x,x^*)\big(r_4+r_5\big)
+\alpha\varepsilon\Big)V,
\end{eqnarray*}
where $\alpha$ is as in \eqref{supremum}. Now by Remark \ref{remark1}, we get
\begin{eqnarray*}
\rho(T^{n_k+1}x,Tx^*)&<&a_1(T^{n_k}x,x^*)r_1+a_2(T^{n_k}x,x^*)r_2+a_3(T^{n_k}x,x^*)r_3
+a_4(T^{n_k}x,x^*)\big(r_4+r_5\big)+\alpha\varepsilon\cr\\[-.3cm]
&=&a_1(T^{n_k}x,x^*)\rho(T^{n_k}x,x^*)+a_2(T^{n_k}x,x^*)\rho(T^{n_k}x,T^{n_k+1}x)\cr\\[-.3cm]
&\ &+\,a_3(T^{n_k}x,x^*)\rho(x^*,Tx^*)
+a_4(T^{n_k}x,x^*)\big(\rho(T^{n_k}x,Tx^*)+\rho(x^*,T^{n_k+1}x)\big)+\alpha\varepsilon.
\end{eqnarray*}
Since $\varepsilon>0$ was arbitrary, we obtain
\begin{eqnarray*}
\rho(T^{n_k+1}x,Tx^*)&\leq&a_1(T^{n_k}x,x^*)\rho(T^{n_k}x,x^*)+a_2(T^{n_k}x,x^*)\rho(T^{n_k}x,T^{n_k+1}x)\cr\\[-.3cm]
&\ &+\,a_3(T^{n_k}x,x^*)\rho(x^*,Tx^*)
+a_4(T^{n_k}x,x^*)\big(\rho(T^{n_k}x,Tx^*)+\rho(x^*,T^{n_k+1}x)\big)\cr\\[-.3cm]
&\leq&a_1(T^{n_k}x,x^*)\rho(T^{n_k}x,x^*)+a_2(T^{n_k}x,x^*)\rho(T^{n_k}x,T^{n_k+1}x)\cr\\[-.3cm]
&\
&+\,a_3(T^{n_k}x,x^*)\big(\rho(x^*,T^{n_k+1}x)+\rho(T^{n_k+1}x,Tx^*)\big)\cr\\[-.3cm]
&\
&+\,a_4(T^{n_k}x,x^*)\big(\rho(T^{n_k}x,T^{n_k+1}x)+\rho(T^{n_k+1}x,Tx^*)+\rho(x^*,T^{n_k+1}x)\big).
\end{eqnarray*}
Therefore,
\begin{eqnarray*}
\rho(T^{n_k+1}x,Tx^*)&\leq&\frac1{1-a_3(T^{n_k}x,x^*)-a_4(T^{n_k}x,x^*)}\Big(a_1(T^{n_k}x,x^*)
\rho(T^{n_k}x,x^*)\cr\\
&\
&+\,\big(a_2(T^{n_k}x,x^*)+a_4(T^{n_k}x,x^*)\big)\rho(T^{n_k}x,T^{n_k+1}x)\cr\\
&\ &+\,
\big(a_3(T^{n_k}x,x^*)+a_4(T^{n_k}x,x^*)\big)\rho(T^{n_k+1}x,x^*)\Big)\cr\\
&\leq&\frac1{1-\alpha}\Big(\alpha\rho(T^{n_k}x,x^*)+\alpha\rho(T^{n_k}x,T^{n_k+1}x)
+\alpha\rho(T^{n_k+1}x,x^*)\Big)\cr\\
&=&\frac\alpha{1-\alpha}\Big(\rho(T^{n_k}x,x^*)+\rho(T^{n_k}x,T^{n_k+1}x)
+\rho(T^{n_k+1}x,x^*)\Big).
\end{eqnarray*}
Consequently, from $T^nx\rightarrow x^*$, there exists a $k_0>0$ such that
$$(T^{n_kx},x^*)\in\frac{1-\alpha}{3\alpha}\cdot V,\quad(T^{n_kx},T^{n_k+1}x)\in\frac{1-\alpha}{3\alpha}\cdot V,\quad\mbox{and}\quad(T^{n_k+1x},x^*)\in\frac{1-\alpha}{3\alpha}\cdot V,$$
for all $k\geq k_0$. Therefore,
$$\rho(T^{n_k+1}x,Tx^*)<\frac\alpha{1-\alpha}\Big(\frac{1-\alpha}{3\alpha}
+\frac{1-\alpha}{3\alpha}+\frac{1-\alpha}{3\alpha}\Big)=1\qquad(k\geq k_0),$$ that
is, $(T^{n_k+1}x,Tx^*)\in V$ for all $k\geq k_0$. Since $V\in\mathscr V$ was arbitrary, it is
seen that $T^{n_k+1}\rightarrow Tx^*$. On the other hand, since $T^{n_k+1}\rightarrow x^*$ and  $X$ is
separated, we must have
$x^*=Tx^*$, and therefore  $x^*$ is a fixed point for $T$.\par To see that
$x^*$ is the unique fixed point for $T$ whenever (i) and (ii) are
satisfied, let $y^*\in X$ be a fixed point for $T$. If
$V\in\mathscr V$, then we consider the following two cases to show
that $(x^*,y^*)\in V$:\vspace{2mm}\par {\bf Case 1:
$\boldsymbol{(x^*,y^*)}$ is an edge of $\boldsymbol G$.}\par Let
$\rho$ be the Minkowski's pseudometric of $V$. Take any arbitrary
$\varepsilon>0$ and write $\rho(x^*,y^*)=r$. Then
$(x^*,y^*)\in(r+\varepsilon)V$ and so by (C2) and Lemma
\ref{lem1}, we have
\begin{eqnarray*}
(x^*,y^*)=(Tx^*,Ty^*)&\in&a_1(x^*,y^*)(r+\varepsilon)V\circ
a_2(x^*,y^*)(r+\varepsilon)V\circ a_2(x^*,y^*)(r+\varepsilon)V\cr\\[-.3cm]&\
&\circ\, a_4(x^*,y^*)(r+\varepsilon)V\circ
a_4(x^*,y^*)(r+\varepsilon)V\cr\\[-.3cm]
&\subseteq&\Big(\big(a_1(x^*,y^*)+2a_2(x^*,y^*)+2a_4(x^*,y^*)\big)r\cr\\[-.3cm]
&\ &+\,\big(a_1(x^*,y^*)+2a_2(x^*,y^*)+2a_4(x^*,y^*)\big)\varepsilon\Big)V\cr\\[-.3cm]
&\subseteq&(\alpha r+\alpha\varepsilon)V.
\end{eqnarray*}
Therefore,
$$\rho(x^*,y^*)<\alpha
r+\alpha\varepsilon=\alpha\rho(x^*,y^*)+\alpha\varepsilon.$$ Since
$\varepsilon>0$ was arbitrary, we get
$\rho(x^*,y^*)\leq\alpha\rho(x^*,y^*)$, and since $\alpha<1$, it
follows that $\rho(x^*,y^*)=0$, that is, $(x^*,y^*)\in
V$.\vspace{2mm}\par {\bf Case 2: $\boldsymbol{(x^*,y^*)}$ is not an edge of $\boldsymbol G$.}\par In this case, by (ii),
there exists a $z\in X$ such that $(x^*,z),(y^*,z)\in E(\widetilde
G)$. Pick a $W\in\mathscr V$ such that $W\circ W\subseteq V$ and denote by $\rho'$, the Minkowski's pseudometric of $W$ and let $n\geq1$. By Proposition \ref{tilde}, $T$ preserves the edges of $\widetilde G$. So
$(x^*,T^nz)=(T^nx^*,T^nz)\in E(\widetilde G)$.
Now write
$$\rho'(x^*,T^{n-1}z)=r_1,\quad\rho'(T^{n-1}z,T^nz)=r_2,\quad\mbox{and}\quad\rho'(x^*,T^nz)=r_3.$$
Then, clearly,
$$(x^*,T^{n-1}z)\in(r_1+\varepsilon)V,\quad(x^*,x^*)\in\varepsilon
V,\quad(T^{n-1}z,T^nz)\in(r_2+\varepsilon)V,$$
$$(x^*,T^nz)\in(r_3+\varepsilon)V,\quad\mbox{and}\quad(T^{n-1}z,x^*)\in(r_1+\varepsilon)V.$$
Therefore, from (C2) and Lemma \ref{lem1}, we have
\begin{eqnarray*}
(x^*,T^nz)=(T^nx^*,T^nz)&\in&a_1(x^*,T^{n-1}z)(r_1+\varepsilon)V\circ
a_2(x^*,T^{n-1}z)\varepsilon V\circ
a_2(x^*,T^{n-1}z)(r_2+\varepsilon)V\cr\\[-.3cm]
&\ &\circ\,a_4(x^*,T^{n-1}z)(r_3+\varepsilon)V\circ
a_4(x^*,T^{n-1}z)(r_1+\varepsilon)V\cr\\[-.3cm]
&\subseteq&\Big(\big(a_1(x^*,T^{n-1}z)+a_4(x^*,T^{n-1}z)\big)r_1+a_2(x^*,T^{n-1}z)r_2+a_4(x^*,T^{n-1}z)r_3\cr\\[-.3cm]
&\
&+\,\big(a_1(x^*,T^{n-1}z)+2a_2(x^*,T^{n-1}z)+2a_4(x^*,T^{n-1}z)\big)\varepsilon\Big)V\cr\\[-.3cm]
&\subseteq&\Big(\big(a_1(x^*,T^{n-1}z)+a_4(x^*,T^{n-1}z)\big)r_1\cr\\[-.3cm]
&\
&+\,a_2(x^*,T^{n-1}z)r_2+a_4(x^*,T^{n-1}z)r_3+\alpha\varepsilon\Big)V.
\end{eqnarray*}
Hence by Remark \ref{remark1},
\begin{eqnarray*}
\rho'(x^*,T^nz)&<&\big(a_1(x^*,T^{n-1}z)+a_4(x^*,T^{n-1}z)\big)r_1+a_2(x^*,T^{n-1}z)r_2
+a_4(x^*,T^{n-1}z)r_3+\alpha\varepsilon\cr\\[-.3cm]
&=&\big(a_1(x^*,T^{n-1}z)+a_4(x^*,T^{n-1}z)\big)\rho'(x^*,T^{n-1}z)\cr\\[-.3cm]
&\
&+\,a_2(x^*,T^{n-1}z)\rho'(T^{n-1}z,T^nz)+a_4(x^*,T^{n-1}z)\rho'(x^*,T^nz)
+\alpha\varepsilon.
\end{eqnarray*}
Since $\varepsilon>0$ was arbitrary, we obtain
\begin{eqnarray*}
\rho'(x^*,T^nz)&\leq&\big(a_1(x^*,T^{n-1}z)+a_4(x^*,T^{n-1}z)\big)\rho'(x^*,T^{n-1}z)\cr\\[-.3cm]
&\
&+\,a_2(x^*,T^{n-1}z)\rho'(T^{n-1}z,T^nz)+a_4(x^*,T^{n-1}z)\rho'(x^*,T^nz)\cr\\[-.3cm]
&\leq&\big(a_1(x^*,T^{n-1}z)+a_2(x^*,T^{n-1}z)+a_4(x^*,T^{n-1}z)\big)\rho'(x^*,T^{n-1}z)\cr\\[-.3cm]
&\
&+\,\big(a_2(x^*,T^{n-1}z)+a_4(x^*,T^{n-1}z)\big)\rho'(x^*,T^nz),
\end{eqnarray*}
which accompanied with \eqref{abc} yields
\begin{eqnarray*}
\rho'(x^*,T^nz)&\leq&\frac{a_1(x^*,T^{n-1}z)+a_2(x^*,T^{n-1}z)+a_4(x^*,T^{n-1}z)}
{1-a_2(x^*,T^{n-1}z)-a_4(x^*,T^{n-1}z)}\cdot\rho(x^*,T^{n-1}z)\cr\\
&<&\alpha\rho'(x^*,T^{n-1}z)=\alpha\rho'(T^{n-1}x^*,T^{n-1}z)<\cdots<\alpha^n\rho'(x^*,z).
\end{eqnarray*}
Similarly, one can show that
$$\rho'(y^*,T^nz)\leq\alpha^n\rho'(y^*,z).$$
Now, for sufficiently large $n$, we have
$\alpha^n\rho'(x^*,z)<1$ and $\alpha^n\rho'(y^*,z)<1$, that is,
$(x^*,T^nz),(y^*,T^nz)\in W$. Therefore, $(x^*,y^*)\in W\circ
W\subseteq V$.\vspace{2mm}\par Consequently, in both cases, we
have $(x^*,y^*)\in V$. Since $V\in\mathscr V$ was arbitrary and
$X$ is separated, it follows that $y^*=x^*$.
\end{proof}

Setting $G=G_0$ and $G=G_1$ in Theorem \ref{main}, we get the next results in uniform
spaces and partially ordered uniform spaces, respectively.  Note
that Corollary \ref{cor1} is a generalization of \cite[2.5.
Theorem]{cir}.

\begin{cor}\label{cor1}
Let the uniform space $X$ be sequentially complete and separated
and $T:X\rightarrow X$ be a \'Ciri\'c-contraction. Then for each
$x\in X$, the sequence $\{T^nx\}$ converges to a fixed point of
$T$. Moreover, if $a_2$ and $a_3$ in {\rm(C2)} coincide on
$X\times X$, then this fixed point is unique, i.e., there exists a
unique $x^*\in\fix(T)$ such that $\{T^nx\}$ converges to $x^*$ for
all $x\in X$.
\end{cor}

\begin{cor}
Let $\preceq$ be a partial order on the sequentially complete and
separated uniform space $X$ satisfying the following property:
\begin{description}
\item[\hspace{5mm}] If a nondecreasing sequence $\{x_n\}$
converges to some point $x\in X$, then it contains a subsequence
$\{x_{n_k}\}$ such that $x_{n_k}\preceq x$ for all $k\geq1$.
\end{description}
Then a nondecreasing order \'Ciri\'c-contraction $T:X\rightarrow
X$ has a fixed point if and only if there exists an $x_0\in X$
such that $x_0\preceq Tx_0$. Moreover, this fixed point is unique
if
\begin{description}
\item[\rm i)] the functions $a_2$ and $a_3$ in {\rm(C2)} coincide
on $X\times X$; and \item[\rm ii)] each two elements of $X$ has
either a lower or an upper bound.
\end{description}
\end{cor}

Our next result is a generalization of the fixed point theorem for
Hardy and Rogers-type contraction \cite{har} from metric spaces to
uniform spaces endowed with a graph. It also generalizes Banach,
Kannan and Chatterjea contractions provided that $0V=\Delta(X)$.

\begin{cor}
Suppose that the uniform space $X$ is sequentially complete and
separated, and satisfies the following properties:
\begin{itemize}
\item If a sequence $\{x_n\}$ converges to some point $x\in X$ and
it satisfies $(x_n,x_{n+1})\in E(G)$ for all $n\geq1$, then there
exists a subsequence $\{x_{n_k}\}$ of $\{x_n\}$ such that
$(x_{n_k},x)\in E(G)$ for all $k\geq1$; \item For all $x,y\in X$,
there exists a $z\in X$ such that $(x,z),(y,z)\in E(\widetilde
G)$.
\end{itemize}
Let $T:X\rightarrow X$ be an edge preserving self-mapping satisfying
the following contractive condition:
\begin{description}
\item[\hspace{5mm}] For all $x,y\in X$ and all
$V_1,V_2,V_3,V_4,V_5\in\mathscr V$,
$$(x,y)\in E(G)\cap V_1,\ (x,Tx)\in V_2,\ (y,Ty)\in V_3,\ (x,Ty)\in V_4,\ \mbox{and}\ (y,Tx)\in V_5$$ imply
$$(Tx,Ty)\in aV_1\circ bV_2\circ bV_3\circ cV_4\circ cV_5,$$
where $a$, $b$ and $c$ are positive real numbers such that
$a+2b+2c<1$.
\end{description}
Then $T$ has a unique fixed point if and only if
$X_T\neq\emptyset$.
\end{cor}

\begin{rem}\rm In \cite{boj12}, Bojor established some results on  the existence
and uniqueness of fixed points for self-mappings $T$, called \'Ciri\'c-Reich-Rus operators, on a  metric space $X$ endowed with a $T$-connected graph $G$ satisfying
$$d(Tx,Ty)\leq ad(x,y)+bd(x,Tx)+cd(y,Ty)\qquad(x,y\in X),$$
where $a,b,c\geq0$ and $a+b+c<1$. Let us review the notion of $T$-connectedness introduced by Bojor:  Let $X$ be a metric space endowed with a graph $G$ (see, \cite[Section 2]{jac}) and $T$ be a self-mapping on $X$. Then $G$ is said to be
$T$-connected if for all $x,y\in X$ with $(x,y)\notin E(G)$, there
exists a finite sequence $(x_i)_{i=0}^N$ of vertices of $G$ such
that
$$x_0=x,\ x_N=y,\ (x_{i-1},x_i)\in E(G)\ \mbox{for}\ i=1,\ldots,N,\
\mbox{and}\ x_i\in X_T\ \mbox{for}\ i=1,\ldots,N-1.$$ \par Note if the graph
$G$ is $T$-connected and $\preceq$ is the partial order on $X$ induced by $G$, then for all $x,y\in
X$ with $x\npreceq y$, there exists a finite sequence
$(x_i)_{i=0}^N$ of vertices of $G$ such that $x_{i-1}\preceq x_i$
for $i=1,\ldots,N$. Hence by the transitivity of $\preceq$, we get
$x\preceq y$, which is impossible. Therefore, the graph $G$ is
$T$-connected if and only if each two elements of $X$ are
comparable. More generally, if a graph $G$ is transitive, then $G$ is $T$-connected if and only if it is complete, which is a  strong condition.
 Hence, Bojor's work can not generalize the same results from partially
ordered metric spaces to metric spaces endowed with a graph.
%we have never used $T$-connectedness in this paper and our results
%are very similar to usual results in partially ordered metric
%spaces.
\par In this work we do not impose the assumption $T$-connectedness on our results and since the \'Ciri\'c-$G$-contractions given here are generalization of  \'Ciri\'c-Reich-Rus operators, therefore  our results generalize the fixed point results  for \'Ciri\'c-Reich-Rus operators in partially ordered metric and uniform spaces.
% and somehow generalize Bojor's results.$\spadesuit$
\end{rem}

\end{document}